\documentclass[12pt,reqno]{amsart}

\usepackage{amssymb}

\newtheorem{theorem}{Theorem}[section]
\newtheorem{proposition}[theorem]{Proposition}
\newtheorem{lemma}[theorem]{Lemma}
\newtheorem{corollary}[theorem]{Corollary}

\theoremstyle{definition}
\newtheorem{definition}[theorem]{Definition}
\newtheorem{remark}[theorem]{Remark}

\numberwithin{equation}{section}

\newcommand{\PP}{\ensuremath{\mathbb{P}}}

\begin{document}

\baselineskip=15pt

\title[Nonabelian Hodge theory for class $\mathcal C$ manifolds]{Nonabelian Hodge
theory for Fujiki class $\mathcal C$ manifolds}

\author[I. Biswas]{Indranil Biswas}

\address{School of Mathematics, Tata Institute of Fundamental
Research, Homi Bhabha Road, Mumbai 400005, India}

\email{indranil@math.tifr.res.in}

\author[S. Dumitrescu]{Sorin Dumitrescu}

\address{Universit\'e C\^ote d'Azur, CNRS, LJAD, France}

\email{dumitres@unice.fr}

\subjclass[2010]{32G13, 53C07, 58D27, 14E05}

\keywords{Nonabelian Hodge theory, flat connection, Higgs bundle, Fujiki class $\mathcal C$ manifold}

\date{}
\begin{abstract}
The nonabelian Hodge correspondence (also known as the Corlette-Simpson 
correspondence), between the polystable Higgs bundles with vanishing Chern classes
on a compact K\"ahler manifold $X$ and the completely reducible flat connections on
$X$, is extended to the Fujiki class $\mathcal C$ manifolds.
\end{abstract}

\maketitle

\section{Introduction}

Given a compact K\"ahler manifold $X$, foundational works of Simpson and Corlette, \cite{Si1}, 
\cite{Co} establish a natural equivalence between the category of local systems over $X$ and
the category of certain 
analytical objects called {\it Higgs bundles} that consist of a holomorphic vector bundle $V$ 
over $X$ together with a holomorphic section $ \theta\, \in\, H^0(X,\, 
\text{End}(V)\otimes\Omega^1_X)$ such that the section $\theta\bigwedge\theta \, \in\, 
H^0(X,\, \text{End}(V)\otimes\Omega^2_X)$ vanishes identically (see also \cite{Si2}). Such a
section $\theta$ is called a {\it Higgs field} on $V$.

While local systems are topological objects on $X$, which correspond to the flat vector bundles 
(or, equivalently, correspond to the equivalence classes of representations of the fundamental group of 
$X$), the Higgs bundles on $X$ are holomorphic objects. There are notions of stability and 
polystability for Higgs bundles which are analogous to the corresponding notions for 
holomorphic vector bundles on $X$, with the difference being that the class of subsheaves
is restricted to only those that are 
invariant under the Higgs field (see Section \ref{s2}). So the (semi)stability condition 
generalizes the (semi)stability of holomorphic vector bundles introduced by Mumford
in the context of geometric invariant theory which he developed.

A more precise statement of the above mentioned equivalence of categories between Higgs bundles and local
systems on $X$ says that there is a natural correspondence between the
completely reducible local systems on $X$ and the
polystable Higgs bundles on $X$ with vanishing rational Chern classes. It may be mentioned that
for polystable Higgs bundles, the vanishing of the first
two rational Chern classes implies the vanishing of all rational Chern classes. This correspondence
is constructed via a Hermitian metric on $V$ that satisfies the Yang--Mills--Higgs equation for
a polystable $(V,\, \theta)$, \cite{Si1}, and a harmonic metric on a vector bundle on $X$
equipped with a completely reducible flat connection \cite{Co}.
The construction of these canonical metrics can be seen as a vast 
generalization of Hodge Theorem on existence of harmonic forms, and for this reason
the above correspondence is also called a ``nonabelian Hodge theorem''.

The aim here is to extend this Corlette-Simpson (nonabelian Hodge) correspondence 
to the more general context of compact Fujiki class $\mathcal C$ manifolds;
see Theorem \ref{thm2}.

Recall that a manifold $M$ is in Fujiki class $\mathcal C$ if it is the image of 
a K\"ahler manifold through a holomorphic map \cite{Fu2}, or, equivalently, $M$ is bimeromorphic to a 
compact K\"ahler manifold \cite{Va} (see Section \ref{s2.2}). The proof of Theorem \ref{thm2} 
uses a well--known functoriality property of Corlette-Simpson correspondence (see Theorem 
\ref{thm1}) and a descent result (see Proposition 
\ref{propositionpullback}) which is inspired by Theorem 1.2 in \cite{GKPT}.

\section{Representations, Higgs bundles and Fujiki class $\mathcal C$ manifolds}

\subsection{Nonabelian Hodge theory}\label{s2}

Let $X$ be a compact connected complex manifold. Fix a base point $x_0\, \in\, X$ to
define the fundamental group $\pi_1(X,\, x_0)$ of $X$. Take a positive integer $r$,
and consider any homomorphism
$$
\rho\, :\, \pi_1(X,\, x_0)\, \longrightarrow\, \text{GL}(r,{\mathbb C})\, .
$$
The homomorphism $\rho$ is called \textit{irreducible} if the standard action of $\rho(\pi_1(X,\, x_0))$
on ${\mathbb C}^r$ does not preserve any nonzero proper subspace of ${\mathbb C}^r$.
The homomorphism $\rho$ is called \textit{completely reducible} if it is a direct sum
of irreducible representations.

Two homomorphisms $\rho_1,\, \rho_2\, :\, \pi_1(X,\, x_0)\, \longrightarrow\, \text{GL}(r,{\mathbb C})$ 
are called \textit{equivalent} if there is an element $g\, \in\, \text{GL}(r,{\mathbb C})$ such
that
$$
\rho_1(\gamma)\,=\, g^{-1} \rho_2(\gamma)g
$$
for all $\gamma\, \in\, \pi_1(X,\, x_0)$. Clearly, this equivalence relation preserves irreducibility
and complete reducibility. The space of equivalence classes of completely reducible
homomorphisms from $\pi_1(X,\, x_0)$ to $\text{GL}(r,{\mathbb C})$ has the structure of an affine
scheme defined over $\mathbb C$, which can be seen as follows. Since $X$ is compact, $\pi_1(X,\, x_0)$ is a finitely
presented group; $\text{GL}(r,{\mathbb C})$ being an affine algebraic group, the space of all homomorphisms
$\text{Hom}(\pi_1(X,\, x_0),\, \text{GL}(r,{\mathbb C}))$ is a complex affine scheme. The adjoint action
of $\text{GL}(r,{\mathbb C})$ on $\text{Hom}(\pi_1(X,\, x_0),\, \text{GL}(r,{\mathbb C}))$ produces an action
of $\text{GL}(r,{\mathbb C})$ on $\text{Hom}(\pi_1(X,\, x_0),\, \text{GL}(r,{\mathbb C}))$. The geometric
invariant theoretic quotient $\text{Hom}(\pi_1(X,\, x_0),\, \text{GL}(r,{\mathbb C}))/\!\!/\text{GL}(r,{\mathbb C})$
is the moduli space of equivalence classes of completely reducible
homomorphisms from $\pi_1(X,\, x_0)$ to $\text{GL}(r,{\mathbb C})$; see \cite{Si3}, \cite{Si4}.
Let ${\mathcal R}(X, r)$ denote this moduli space of equivalence classes of completely reducible
homomorphisms from $\pi_1(X,\, x_0)$ to $\text{GL}(r,{\mathbb C})$. It is known as the
Betti moduli space.

A homomorphism $\rho\, :\, \pi_1(X,\, x_0)\, \longrightarrow\, \text{GL}(r,{\mathbb C})$ 
produces a holomorphic vector bundle $E$ on $X$ of rank $r$ equipped with a flat holomorphic 
connection, together with a trivialization of the fiber $E_{x_0}$. Equivalence classes of 
such homomorphisms correspond to holomorphic vector bundles of rank $r$ equipped with a flat 
holomorphic connection; this is an example of Riemann--Hilbert correspondence. A connection 
$\nabla$ on a vector bundle $E$ is called irreducible if there is no subbundle $0\, \not=\, 
F\, \subsetneq E$ such that $\nabla$ preserves $F$. A connection $\nabla$ on a vector bundle 
$E$ is called completely reducible if
$$
(E,\,\nabla)\,=\, \bigoplus_{i=1}^N(E_i,\,\nabla^i)\, ,
$$
where each $\nabla^i$ is an irreducible connection on $E_i$. We note that irreducible (respectively,
completely reducible) flat connections of rank $r$ on $X$
correspond to irreducible (respectively, completely reducible)
equivalence classes of homomorphisms from $\pi_1(X,\, x_0)$ to $\text{GL}(r,{\mathbb C})$.

A \textit{Higgs field} on a holomorphic vector bundle $V$ on $X$ is a holomorphic section
$$
\theta\, \in\, H^0(X,\, \text{End}(V)\otimes\Omega^1_X)
$$
such that the section $\theta\bigwedge\theta \, \in\, H^0(X,\, \text{End}(V)\otimes\Omega^2_X)$
vanishes identically \cite{Si1}, \cite{Si2}.
If $(z_1, \,\ldots,\, z_d)$ are local holomorphic coordinates on $X$ with respect to which the 
local expression of the section $\theta$ is $\sum_i \theta_i\otimes dz_i$, with $\theta_i$ being 
locally defined holomorphic endomorphisms of $V$, the above integrability condition $\theta\bigwedge\theta 
\,=\,0$ is equivalent to the condition that $[\theta_i,\, \theta_j]\,=\, 0$ for all $i,\, j$.

A \textit{Higgs bundle} on $X$ is a holomorphic vector bundle on $X$
together with a Higgs field on it. A homomorphism of Higgs bundles $(V_1,\, \theta_1)\, \longrightarrow\,
(V_2,\, \theta_2)$ is a holomorphic homomorphism
$$
\Psi\, :\, V_1\, \longrightarrow\, V_2
$$
such that $\theta_2\circ\Psi\,=\, (\Psi\otimes \text{Id}_{\Omega^1_X})\circ\theta_1$ as homomorphisms from
$V_1$ to $V_2\otimes\Omega^1_X$.

Assume now that $X$ is K\"ahler, and fix a K\"ahler form $\omega$ on $X$. The \textit{degree} 
of a torsionfree coherent analytic sheaf $F$ on $X$ is defined to be
$$
\text{degree}(F)\, :=\, \int_X c_1(\det F)\wedge \omega^{d-1}\, \in\, \mathbb R\, ,
$$
where $d\,=\, \dim_{\mathbb C} X$; see \cite[Ch.~V, \S~6]{Ko} (also Definition 1.34 in \cite{Br})
for determinant line bundle $\det F$. The number
$$
\mu(F)\, :=\, \frac{\text{degree}(F)}{\text{rank}(F)}\, \in\, \mathbb R
$$
is called the \textit{slope} of $F$.

A Higgs bundle $(V,\, \theta)$ on $X$ is called \textit{stable} (respectively, \textit{semistable})
if for every coherent analytic subsheaf $F\, \subset\, V$ with $0\, <\, \text{rank}(F)\,<\,
\text{rank}(V)$ and $\theta(F)\, \subset\, F\otimes\Omega^1_X$, the inequality 
$$
\mu(F)\, <\, \mu(V) \ \ {\rm (respectively,}\ \mu(F)\, \leq\, \mu(V){\rm )}
$$
holds. A Higgs bundle $(V,\, \theta)$ is called \textit{polystable} if it is semistable and
a direct sum of stable Higgs bundles. To verify the stability (or semistability) condition it suffices to
consider coherent analytic subsheaves $F\, \subsetneq\, V$ such that the quotient $V/F$ is
torsionfree \cite[Ch.~V, Proposition 7.6]{Ko}. These subsheaves are 
reflexive (see \cite[Ch.~V, Proposition 5.22]{Ko}).

\begin{theorem}[{\cite{Si1}, \cite{Co}, \cite{Si2}}]\label{thm0}
There is a natural equivalence of categories between the following two:
\begin{enumerate}
\item The objects are completely reducible flat complex connections on $X$, and morphisms are connection
preserving homomorphisms.

\item Objects are polystable Higgs bundles $(V,\, \theta)$ on $X$
such that $c_1(V)\,=\, 0\, =\, c_2(V)$, where $c_i$ denotes the $i$--th Chern class with coefficient in
$\mathbb Q$; the morphisms are homomorphisms of Higgs bundles.
\end{enumerate}
\end{theorem}

In \cite{Si2}, the conditions on the Chern classes of the polystable Higgs bundle $(V,\, \theta)$ are
$\text{degree}(V)\,=\,0\, =\, (ch_2(V)\cup [\omega^{d-2}])\cap [X]$, instead of the above conditions
$c_1(V)\,=\, 0\, =\, c_2(V)$. However, since the existence of flat connection
on a complex vector bundle implies that all its rational Chern classes vanish, these two sets of
conditions are equivalent in the given context.

\begin{remark}\label{rem1}
Recall that the notion of (poly)stability depends on the choice of the K\"ahler class
of $\omega$. In the case of vanishing first two Chern classes, these notions are actually independent
of the class of $\omega$. In fact, given an equivalence class of completely reducible 
homomorphism $\rho\, :\, \pi_1(X,\, x_0)\, \longrightarrow\, \text{GL}(r,{\mathbb C})$, the 
Higgs bundle $(V,\, \theta)$ of rank $r$ associated to $\rho$ by the equivalence of 
categories in Theorem \ref{thm0} is in fact independent of the choice of $\omega$. Indeed, the
local system $\rho$ is obtained from $(V,\, \theta)$ by constructing a Hermitian metric $h$ on
$V$ that satisfies the Yang--Mills--Higgs equation
$K(\nabla^h) + \lbrack \theta ,\, \theta_h^* \rbrack\,=\,0$, where $K(\nabla^h)$ is the
curvature of the Chern connection $\nabla^h$ on $V$ corresponding to $h$, and $\theta_h^*$ is the
adjoint of $\theta$ with respect to $h$. Since this Yang--Mills--Higgs equation does not 
depend on the K\"ahler form $\omega$, the flat connection corresponding to $(V,\, \theta)$
is independent of $\omega$.
\end{remark}

We recall a basic property of the equivalence of categories in Theorem \ref{thm0}.

\begin{theorem}[{\cite{Si2}}]\label{thm1}
Let $X$ and $X_1$ be compact connected K\"ahler manifolds, and let
$$\beta\, :\, X_1\, \longrightarrow\, X$$
be any holomorphic map. Let $(E,\, \nabla)$ be a completely reducible flat connection
on $X$, and let $(V,\, \theta)$ be the corresponding polystable Higgs bundle on $X$
with $c_1(V)\,=\, 0\, =\, c_2(V)$. Then the pulled back Higgs bundle $(\beta^*V,\, \beta^*\theta)$
is polystable, and, moreover, the flat connection corresponding to it coincides with
$(\beta^*E,\, \beta^*\nabla)$.
\end{theorem}

To explain Theorem \ref{thm1}, take any completely reducible homomorphism
$\rho\, :\, \pi_1(X,\, x_0)\, \longrightarrow\,\text{GL}(r, {\mathbb C})$.
Let $\alpha\, :\, \widetilde{X}\, \longrightarrow\, \text{GL}(r, {\mathbb C})/{\rm U}(r)$ be
the $\rho$-equivariant harmonic map defined on the universal cover of the K\"ahler
manifold $(X,\, x_0)$. Then $\alpha\circ\widetilde{\beta}$ is $\beta^*\rho$-equivariant
harmonic, where $\widetilde \beta$ is the lift of the map $\beta$ in Theorem \ref{thm1}
to a universal covering of $X_1$. Also, as noted in Remark \ref{rem1}, the Yang--Mills--Higgs equation
for $(V,\,\theta)$ does not depend on the K\"ahler form $\omega$. Theorem \ref{thm1}
follows from these facts.

Note that Remark \ref{rem1} can be seen as a particular case of Theorem \ref{thm1} by setting 
$\beta$ to be the identity map of $X$ equipped with two different K\"ahler structures.

A useful particular case of Theorem \ref{thm1} is the following:

\begin{corollary}\label{cor1}
Let the map $\beta$ in Theorem \ref{thm1} be such that the corresponding homomorphism of fundamental groups
$$
\beta_*\, :\, \pi_1(X_1,\, x_1)\, \longrightarrow\, \pi_1(X,\, \beta(x_1))
$$
is trivial. For any polystable Higgs bundle $(V,\, \theta)$ on $X$ of rank $r$ with $c_1(V)\,=\, 0\, =\, c_2(V)$,
\begin{enumerate}
\item $\beta^*V\,=\, {\mathcal O}^{\oplus r}_{X_1}$, and

\item $\beta^*\theta\,=\, 0$.
\end{enumerate}
\end{corollary}

The equivalence of categories in Theorem \ref{thm0}, between the equivalence classes of completely reducible 
flat connections on $X$ and the polystable Higgs bundles $(V,\, \theta)$ on $X$ such that 
$c_1(V)\,=\, 0\, =\, c_2(V)$, extend to the context of principal $G$--bundles, where $G$ is 
any complex affine algebraic group \cite{BG}.

Let ${\mathcal H}(X,r)$ denote the moduli space of polystable Higgs bundles $(V,\, \theta)$ 
of rank $r$ on $X$ such that $c_1(V)\,=\, 0\, =\, c_2(V)$, where $c_i$ denotes the rational 
$i$--th Chern class. It is canonically homeomorphic to the earlier defined moduli space 
${\mathcal R}(X, r)$ (of equivalence classes of completely reducible homomorphisms from 
$\pi_1(X,\, x_0)$ to $\text{GL}(r,{\mathbb C})$). However the complex structure of these
two moduli spaces are different in general.

Now assume that $X$ is a smooth projective variety defined over $\mathbb C$. Simpson proved
the following two results in \cite{Si2}.

\begin{theorem}[{\cite{Si2}}]\label{thms}
There is an equivalence of categories between the following two:
\begin{enumerate}
\item The objects are flat complex connections on $X$, and morphisms are connection
preserving homomorphisms.

\item Objects are semistable Higgs bundles $(V,\, \theta)$ on $X$ such that $c_1(V)
\,=\, 0\, =\, c_2(V)$, where $c_i$ denotes the $i$--th Chern class with coefficient
in $\mathbb Q$; the morphisms are homomorphisms of Higgs bundles.
\end{enumerate}
\end{theorem}

The equivalence of categories in Theorem \ref{thms} extends the one in Theorem \ref{thm0}
but by imposing the condition that $X$ is complex projective.

\begin{proposition}[{\cite{Si2}}]\label{props}
Take $X$ and $X_1$ in Theorem \ref{thm1} and Corollary \ref{cor1} to be 
smooth complex projective varieties. Then Theorem \ref{thm1} and Corollary \ref{cor1}
remain valid if polystability is replaced by semistability.
\end{proposition}

\subsection{Fujiki class $\mathcal C$ manifolds}\label{s2.2}

A compact complex manifold is said to be in {\it the Fujiki class
$\mathcal C$} if it is the image of a compact K\"ahler space under a holomorphic map \cite{Fu2}.
A result of Varouchas, \cite[Section\,IV.3]{Va}, asserts 
that a compact complex manifold $M$ belongs to Fujiki class $\mathcal C$ if and
only if there is a holomorphic map
$$
\phi\, :\, X\, \longrightarrow\, M
$$
such that
\begin{itemize}
\item $X$ is a compact K\"ahler manifold, and

\item $\phi$ is bimeromorphic.
\end{itemize}
In other words, $M$ lies in class $\mathcal C$ if and
only if it admits a compact K\"ahler modification.

\section{A descent result for vector bundles}

A holomorphic line bundle $L$ on a compact complex Hermitian manifold $(Y,\, \omega_Y)$ is called numerically 
effective if for every $\epsilon \, >\, 0$, there is a Hermitian metric $h_\epsilon$ on $L$ such that 
$\text{Curv}(L, h_\epsilon) \, \geq\, -\epsilon \omega_Y$, where $\text{Curv}(L, h_\epsilon)$ is the curvature of 
the Chern connection on $L$ for $h_\epsilon$; since $Y$ is compact, this condition does not depend on the choice of 
the Hermitian metric $\omega_Y$ \cite[Definition 1.2]{DPS}. A holomorphic vector bundle $\mathcal E$ on $Y$ is 
called numerically effective if the tautological line bundle ${\mathcal O}_{{\mathbb P}({\mathcal E})}(1)$ on
${\mathbb P}({\mathcal E})$ is numerically effective \cite[p.~305, Definition 1.9]{DPS}. A holomorphic
vector bundle ${\mathcal E}$ on $Y$ is called numerically flat if both ${\mathcal E}$ and ${\mathcal E}^*$
are numerically effective \cite[p.~311, Definition 1.17]{DPS}.

\begin{proposition}\label{propositionpullback}
Let $\phi\, :\, X\,\longrightarrow\, M$ be a proper bimeromorphic morphism between complex 
manifolds, and let $V\, \longrightarrow \,X$ be a holomorphic vector bundle such that for 
every $x\, \in\, M$, the restriction $V\vert_{\phi^{-1}(x)}$ is a numerically flat vector bundle. Then
there exists a holomorphic vector bundle $W$ on $M$ such that $V \,\simeq\, \phi^* W$.
\end{proposition}

\begin{proof} Set $r\,:=\,{\text rank}(V)$.
We shall prove the proposition in three steps.

{\em Step 1. Assume that $\phi$ is the blow-up of $M$ along a smooth center $Z$.}\, This is 
well-known; for the convenience of the reader we give an argument in the spirit of the method 
of \cite{GKPT} (Sections 4 and 5). Let $E \,\subset\, X$ be the exceptional divisor of the 
blow-up. The restriction $V|_E$ is a vector bundle such that the restriction to every fiber 
of $$\phi|_E\,:\, E\,\longrightarrow \,Z$$ is numerically flat. As the fibers are projective
spaces, it can be shown that the vector bundle $V|_E$ is trivial on the fibers of $\phi|_E$.
Indeed, a numerically flat bundle admits a filtration of holomorphic subbundles such that each
successive quotient admits a unitary flat connection \cite[p.~311, Theorem 1.18]{DPS}. Since a
projective space is simply connected,
each successive quotient is actually trivial. On the other hand, an extension of a trivial bundle
bundle on ${\mathbb C}{\mathbb P}^k$ by a trivial bundle is also trivial, because
$H^1({\mathbb C}{\mathbb P}^k,\, {\mathcal O}_{{\mathbb C}{\mathbb P}^k})\,=\, 0$. 

Since $\phi|_E$ is locally 
trivial, we see that there exists a vector bundle $W_Z$ on $Z$ such that $V|_E \,\simeq\, 
(\phi|_E)^* W_Z$. Consider now the projectivised vector bundle $\pi\, :\, 
\PP(V)\,\longrightarrow\, X$. Then $$\pi^* E \,\simeq \,\PP(V|_E) \,\simeq\, \PP((\phi|_E)^* W_Z)
\,\simeq\, (\phi|_E)^*\PP(W_Z)$$
is a divisor that admits a fibration onto $\PP(W_Z)$. In fact, for any point $z \,\in\, Z$, we have
$$
\PP(V|_{\phi^{-1}(z)}) \,\simeq\, \phi^{-1}(z) \times \PP(W_{Z, z})
$$
and the fibration is given by projection onto the second factor. Since the restriction of the 
divisor $E$ to $\phi^{-1}(z)$ is anti-ample, this also holds for the restriction of $\pi^* 
E$ to the fibers of $\pi^* E\,\longrightarrow\, \PP(W_Z)$. Now we can apply a theorem of Fujiki,
\cite[p.~495, Theorem 2]{Fu1}, to see that there exist a variety $T$ and a bimeromorphic morphism 
$\widetilde{\phi}\, :\, \PP(V)\, \longrightarrow\, T$ such that $\widetilde{\phi}|_{\pi^* E}$
is the fibration $\pi^* E \,\longrightarrow\, \PP(W_Z)$ and the restriction of
it to $\PP(V) \setminus \pi^* E$ is an isomorphism. By 
construction the variety $T$ admits a morphism onto $M$ such that all the fibers are 
isomorphic to $\PP^{r-1}$; in particular, $T$ is smooth and $\PP(V)$ is the blowup of $T$ 
along $\PP(W_Z)$. The push-forward of $c_1(\mathcal O_{\PP(V)}(1))$ onto $T$ defines a 
Cartier divisor on $T$ such that the restriction to the fibers of $T \,\longrightarrow\, M$ 
is the hyperplane class. Thus the corresponding direct image sheaf defines a vector bundle $W 
\,\longrightarrow\, M$ satisfying the condition that $V \,\simeq\, \phi^* W$.

{\em Step 2. Assume that $\phi$ is the composition of smooth blowups.}\,
Set $X_0\,:=\,X$ and $X_k\,:=\,M$, and for $i \,\in \,\{1,\, \cdots, \,k\}$, let 
$\nu_i\, :\, X_{i-1}\,\longrightarrow\, X_i$ be blowups such that
$$
\phi \,= \,\nu_k \circ \cdots \circ \nu_1.
$$
Since every $\nu_1$-fiber is contained in a $\phi$-fiber, it is evident that the restriction
of $V$ to every $\nu_1$-fiber is trivial. Thus, by Step 1, there exists
a vector bundle $V_1$ on $X_1$ such that $V \,\simeq\, \nu_1^* V_1$. 

We shall now proceed by induction on $i \, \in\, \{1,\, \cdots,\, k\}$ and assume that we 
have found a vector bundle $V_i \,\longrightarrow \,X_i$ such that its pull-back to $X$ is 
isomorphic to $V$. We have to check that $V_i$ satisfies the triviality condition with 
respect to the morphism $\nu_{i+1}$: let $G \,\subset\, X_i$ be any $\nu_{i+1}$-fiber and let $Z 
\,\subset\, (\nu_i \circ \ldots \circ \nu_1)^{-1}(G)$ be an irreducible component that surjects 
onto $G$. Since $G$ is contracted by $\nu_{i+1}$, the variety $Z$ is contained in a 
$\phi$-fiber. Consequently, $V|_{Z}$ is trivial. Since
$$
V|_{Z} \,\simeq \,((\nu_i \circ \ldots \circ \nu_1)|_Z)^* (V_i|_G)\, ,
$$
this shows that $V_i|_G$ is numerically flat. Thus we can apply Step 1 to $\nu_{i+1}$.

{\em Step 3. The general case.}\,
If $\nu\,:\, X'\,\longrightarrow\, X$ is a bimeromorphic morphism, then the pull-back $\nu^* 
V$ is a vector bundle on $X'$ that satisfies the assumption with respect to the morphism 
$\phi \circ \nu$. Thus it suffices to prove the statement for $\phi \circ \nu$. Since any 
bimeromorphic morphism between manifolds is dominated by a sequence of blowups with smooth 
centers, we can assume without loss generality that $\phi$ is a composition of blowups with 
smooth centers. Therefore, the proof is completed using Step 2.
\end{proof}

\section{Nonabelian Hodge theory for Moishezon manifolds}

Let $M$ be a compact Moishezon manifold. Recall that Moishezon manifolds, defined as manifolds of maximal 
algebraic dimension, were introduced and studied by Moishezon
in \cite{Mo}, where he proved that they are birational to 
smooth complex projective manifolds \cite{Mo} (see also \cite[p.~26, Theorem~3.6]{Ue}).

\begin{definition}\label{def1}
Take a Higgs bundle $(V,\, \theta)$ on $M$ such
that $c_1(V)\,=\, 0\,=\, c_2(V)$.
The Higgs bundle $(V,\, \theta)$ is called \textit{semistable} (respectively,
\textit{polystable}) if for every
pair $(C,\, \tau)$, where $C$ is a compact connected Riemann surface and $\tau\, :\, C\,
\longrightarrow\, M$ is a holomorphic map, the pulled back Higgs bundle
$(\tau^*V,\, \tau^*\theta)$ is semistable (respectively,
polystable). Clearly, it is enough to consider only nonconstant maps $\tau$.
\end{definition}

When $M$ is a smooth complex projective variety, then semistability and
polystability according to Definition \ref{def1}
coincide with those described in Section \ref{s2}. Indeed, from
Proposition \ref{props} (respectively, Theorem \ref{thm1}) we know that a semistable
(respectively, polystability) Higgs bundle
$(V,\, \theta)$ on $M$ with $c_1(V)\,=\, 0\,=\, c_2(V)$ is semistable (respectively, polystability)
in the sense of Definition \ref{def1}. Conversely, if $(V,\, \theta)$ is a
Higgs bundle on $M$ with $c_1(V)\,=\, 0\,=\, c_2(V)$ such that it is
semistable (respectively, polystability) in the sense of Definition \ref{def1}, then it
is straightforward to deduce that $(V,\, \theta)$ is semistable (respectively, polystability).

\begin{theorem}\label{thmfm}
Let $M$ be a compact Moishezon manifold.
There is an equivalence of categories between the following two:
\begin{enumerate}
\item The objects are flat complex connections on $M$, and morphisms are connection
preserving homomorphisms.

\item Objects are Higgs bundles $(V,\, \theta)$ on $M$
satisfying the following conditions: $c_1(V)\,=\, 0\, =\, c_2(V)$, and $(V,\, \theta)$
is semistable. The morphisms are homomorphisms of Higgs bundles.
\end{enumerate}
\end{theorem}

\begin{proof}
Fix a holomorphic map $$\phi\, :\, X\, \longrightarrow\, M$$ from a smooth
complex projective variety $X$ such that $\phi$ is bimeromorphic.

Let $(E,\, \nabla)$ be a flat complex connection on $M$. Consider the flat complex
connection $(\phi^*E,\, \phi^*\nabla)$ on $X$. Let $(V,\, \theta_V)$ be the semistable
Higgs bundle on $X$ that corresponds to it by Theorem \ref{thms}.

We shall show that there is a holomorphic vector bundle $W$ on
$M$ such that $\phi^*W\, =\, V$.

Let $\beta\,:\, F' \,\longrightarrow\, X$ be the desingularization of a subvariety $F 
\,\subset\, X$ that is contained in a $\phi$-fiber. 
Using the fact that $F$ is contracted by the map $\phi$ we conclude that the homomorphism
$$\beta_*\,:\, \pi_1(F') \,\longrightarrow\, \pi_1(X)$$
induced by $\beta$ is trivial. Since Corollary \ref{cor1} remains valid when polystability is 
replaced by semistability (see Proposition \ref{props}), from Corollary \ref{cor1}(1) we know 
that $\beta^* V$ is a trivial holomorphic vector bundle. In particular, the restriction $V|_F$ is 
numerically flat. Therefore, by Proposition \ref{propositionpullback}, there is a holomorphic 
vector bundle $W$ on $M$ such that $\phi^*W\, =\, V$.

Let $U\, \subset\, M$ be the open subset over which $\phi$ is an isomorphism. The Higgs field 
$\theta_V$ produces a Higgs field on $W\vert_U$. Now by Hartogs' extension theorem, this Higgs field on
$W\vert_U$ extends to a Higgs field on $W$ over $M$; this extended Higgs field on $W$ will be
denoted by $\theta_W$.

We have $c_1(W)\,=\, 0\, =\, c_2(W)$, because $c_1(V)\,=\, 0\, =\, c_2(V)$. We shall
show that the Higgs bundle $(W,\, \theta_W)$ on $M$ is semistable.

Take any pair $(C,\, \tau)$, where $C$ is a compact connected Riemann surface and $\tau\, :\, C\,
\longrightarrow\, M$ is a nonconstant holomorphic map. Then there is a triple
$(\widetilde{C},\, \psi,\, \widetilde{\tau})$ such that
\begin{itemize}
\item $\widetilde C$ is a compact connected Riemann surface,

\item $\psi\, :\, \widetilde{C}\, \longrightarrow\, C$ is a surjective holomorphic map,

\item ${\widetilde\tau}\, :\, {\widetilde C}\, \longrightarrow\, X$ is a holomorphic map, and

\item $\phi\circ{\widetilde\tau}\,=\, \tau\circ\psi$.
\end{itemize}
{}From Theorem \ref{thm1} and Proposition \ref{props} we know that the Higgs bundle
$({\widetilde\tau}^*V,\, {\widetilde\tau}^*\theta_V)$ is semistable. Combining this
with the two facts that $\phi\circ{\widetilde\tau}\,=\, \tau\circ\psi$ and $(V,\, \theta_V)\,=\,
(\phi^*W,\, \phi^*\theta_W)$, we conclude that the Higgs bundle
$(\psi^*\tau^*W,\, \psi^*\tau^*\theta_W)$ is semistable. But this implies that
$(\tau^*W,\, \tau^*\theta_W)$ is semistable. Indeed, if a subbundle $W'\, \subset\, \tau^*W$
contradicts the semistability condition for $(\tau^*W,\, \tau^*\theta_W)$, then
$\psi^*W'\, \subset\, \psi^*\tau^*W$ contradicts the semistability condition for
$(\psi^*\tau^*W,\, \psi^*\tau^*\theta_W)$. Therefore, the Higgs bundle $(W,\, \theta_W)$ is semistable.

To prove the converse, let $(V,\, \theta)$ be a Higgs bundle on $M$
satisfying the following conditions: $c_1(V)\,=\, 0\, =\, c_2(V)$, and $(V,\, \theta)$
is semistable. Consider the Higgs bundle $(\phi^*V,\, \phi^*\theta)$ on $X$. We evidently
have $c_1(\phi^* V)\,=\, 0\, =\, c_2(\phi^* V)$. We shall prove that
$(\phi^*V,\, \phi^*\theta)$ is semistable.

Take any pair pair $(C,\, \tau_1)$, where $C$ is a compact connected Riemann surface and
$\tau_1\, :\, C\, \longrightarrow\, X$ is a holomorphic map. Set
$$
\tau\,=\, \phi\circ\tau_1\, .
$$
Therefore, the given condition that $(\tau^*V,\, \tau^*\theta)$ is semistable implies that
$(\tau^*_1\phi^*V,\, \tau^*_1\phi^*\theta)$ is semistable. But this implies that
$(\phi^*V,\, \phi^*\theta)$ is semistable with respect to any polarization on $X$. Let
$(E',\, \nabla')$ be the flat complex connection on $X$ that corresponds to
$(\phi^*V,\, \phi^*\theta)$ by Theorem \ref{thms}.
Since the map $\phi$ is bimeromorphic, the homomorphism $\phi_*\,:\, 
\pi_1(X) \,\longrightarrow \,\pi_1(M)$ induced by it is an isomorphism. Consequently,
$(E',\, \nabla')$ produces a flat complex connection on $M$.

It is straightforward to check that
the above two constructions, namely from flat connection on $M$ to Higgs bundles on $M$
and vice versa, are inverses of each other.
\end{proof}

\begin{proposition}\label{propfm}
Let $M$ be a compact Moishezon manifold.
There is an equivalence of categories between the following two:
\begin{enumerate}
\item The objects are completely reducible flat complex connections on $M$, and morphisms
are connection preserving homomorphisms.

\item Objects are Higgs bundles $(V,\, \theta)$ on $M$
such that $c_1(V)\,=\, 0\, =\, c_2(V)$, and $(\phi^*V,\, \phi^*\theta)$
is polystable; the morphisms are homomorphisms of Higgs bundles.
\end{enumerate}
\end{proposition}

\begin{proof}
The proof is very similar to the proof of Theorem \ref{thmfm}.

Let $(V,\, \theta)$ be a Higgs bundle on $M$ such that $c_1(V)\,=\, 0\, =\, c_2(V)$ and 
$(\phi^*V,\, \phi^*\theta)$ is polystable. Take any pair pair $(C,\, \tau_1)$, where $C$ is a 
compact connected Riemann surface and $\tau_1\, :\, C\, \longrightarrow\, X$ is a holomorphic 
map. Setting $\tau\,=\, \phi\circ\tau_1$ we conclude that $(\tau^*V,\, \tau^*\theta)\,=\, 
(\tau^*_1\phi^*V,\, \tau^*_1\phi^*\theta)$ is polystable. This implies that $(\phi^*V,\, 
\phi^*\theta)$ is semistable. Let $(E,\, \nabla)$ be the complex flat connection on $X$ that 
corresponds to $(\phi^*V,\, \phi^*\theta)$ by Theorem \ref{thms}. If $\tau_1(C)$ is an intersection of 
very ample hypersurfaces on $X$, the homomorphism of fundamental groups induced by $\tau_1$
$$
\tau_{1*}\, :\, \pi_1(C,\, x_0)\, \longrightarrow\, \pi_1(X,\, x_0)
$$
is surjective. Since $(\tau^*_1\phi^*V,\, \tau^*_1\phi^*\theta)$ is polystable, the restriction
of $(E,\, \nabla)$ to $\tau_1(C)$ is completely reducible. Now from the surjectivity of
$\tau_{1*}$ it follows immediately that $(E,\, \nabla)$ is completely reducible on $X$.
Since the homomorphism $\phi_*\,:\, \pi_1(X) \,\longrightarrow \,\pi_1(M)$ induced by $\phi$ is
an isomorphism, the completely reducible complex flat connection $(E,\, \nabla)$ on $X$
produces a completely reducible complex flat connection on $M$.

To prove the converse, let $(E,\, \nabla)$ be a completely reducible complex flat connection on 
$M$. Since the homomorphism $\phi_*\,:\, \pi_1(X) \,\longrightarrow \,\pi_1(M)$ induced by 
$\phi$ is an isomorphism, the pulled back flat connection $(\phi^*E,\, \phi^*\nabla)$ is 
completely reducible. Let $(V,\, \theta_V)$ be the polystable Higgs bundle on $X$ corresponding 
to $(\phi^*E,\, \phi^*\nabla)$. As in the proof of Theorem \ref{thmfm}, there is a Higgs bundle 
$(W,\, \theta_W)$ on $M$ such that
$$
(\phi^*W,\, \phi^*\theta_W)\, =\, (V,\, \theta_V)
$$
and $c_1(W)\,=\, 0\,=\, c_2(W)$.

In the proof of Theorem \ref{thmfm} it was shown that $(W,\, \theta_W)$ is semistable.
To complete the proof we need to show that $(W,\, \theta_W)$ is polystable.

Take any pair $(C,\, \tau)$, where $C$ is a compact connected Riemann surface and $\tau\, :\, C\,
\longrightarrow\, M$ is a nonconstant holomorphic map. There is a triple
$(\widetilde{C},\, \psi,\, \widetilde{\tau})$ such that
\begin{itemize}
\item $\widetilde C$ is a compact connected Riemann surface,

\item $\psi\, :\, \widetilde{C}\, \longrightarrow\, C$ is a surjective holomorphic map,

\item ${\widetilde\tau}\, :\, {\widetilde C}\, \longrightarrow\, X$ is a holomorphic map, and

\item $\phi\circ{\widetilde\tau}\,=\, \tau\circ\psi$.
\end{itemize}
We know that $({\widetilde\tau}^*\phi^*W,\, {\widetilde\tau}^*\phi^*\theta_W)$ is
polystable and the corresponding flat connection, namely
$({\widetilde\tau}^*\phi^*E,\, {\widetilde\tau}^*\phi^*\nabla)$, is
completely reducible. For the homomorphism of fundamental groups induced by $\psi$
$$
\psi_*\, :\, \pi_1(\widetilde{C})\, \longrightarrow\, \pi_1(C)
$$ 
the image is a finite index subgroup. This implies that the flat connection
$(\tau^*E,\, \tau^*\nabla)$ is completely reducible. Therefore, the Higgs bundle
$(\tau^*W,\, \tau^*\theta_W)$ is polystable, so $(W,\, \theta_W)$ is polystable.
This completes the proof.
\end{proof}

\section{Nonabelian Hodge theory for Fujiki class $\mathcal C$ manifolds}

Let $M$ be a compact connected complex manifold lying in Fujiki class $\mathcal C$. Fix a
bimeromorphic map
\begin{equation}\label{xf}
\phi\, :\, X\, \longrightarrow\, M
\end{equation}
such that $X$ is compact K\"ahler. Let $(V, \,\theta)$ be a Higgs bundle on $M$ such that
$c_1(V)\,=\, 0\,=\, c_2(V)$. Further assume that the pulled back Higgs bundle $(\phi^*V,\, \phi^*\theta)$
is polystable. As noted in Remark \ref{rem1}, this condition is independent of the choice of the
K\"ahler form on $X$.

\begin{lemma}\label{lem1}
Let $f\, :\, Y\, \longrightarrow\, M$ be a holomorphic map from a compact K\"ahler manifold $Y$ such
that $f$ is bimeromorphic. Then the pulled back Higgs bundle $(f^*V,\, f^*\theta)$ is also polystable.
\end{lemma}

\begin{proof}
Consider the irreducible component of the fiber product $Y\times_M X$ that dominates $M$. Let 
$Z$ be a desingularization of it. Let $p_Y$ and $p_X$ be the natural projections of $Z$ to $Y$ and $X$
respectively.

Since $(\phi^*V,\, \phi^*\theta)$ is polystable with $c_1(\phi^*V)\,=\, 0\, =\, c_2(\phi^*V)$, from
Theorem \ref{thm1} we know that $(p^*_X\phi^*V,\, p^*_X\phi^*\theta)$ is polystable; as before, this
condition is independent of the choice of the K\"ahler form on $Z$. The Higgs
bundle $(p^*_Y f^*V,\, p^*_Y f^*\theta)$ is polystable, because
$$
(p^*_X\phi^*V,\, p^*_X\phi^*\theta)\,=\, (p^*_Y f^*V,\, p^*_Y f^*\theta)\, .
$$
Let $(W,\, \nabla)$ be the completely reducible flat complex connection on $Z$.

The homomorphism $p_{Y*}\, :\, \pi_1(Z,\, z_0)\, \longrightarrow\, \pi_1(Y,\, p_Y(z_0))$
induced by $p_Y$ is an isomorphism, because the map $p_Y$ is bimeromorphic. So $(W,\, \nabla)$
gives a completely reducible flat complex connection $(W',\, \nabla')$ on $Y$. The Higgs bundle on $Y$
corresponding to $(W',\, \nabla')$ is isomorphic to $(f^*V,\, f^*\theta)$. In particular,
$(f^*V,\, f^*\theta)$ is polystable.
\end{proof}

From Lemma \ref{lem1} it follows that the second category in the following theorem is independent of
the choice of the pair $(X,\, \phi)$.

\begin{theorem}\label{thm2}
Let $M$ be a compact connected complex manifold lying in Fujiki class $\mathcal C$.
There is an equivalence of categories between the following two:
\begin{enumerate}
\item The objects are completely reducible flat complex connections on $M$, and morphisms are connection
preserving homomorphisms.

\item Objects are Higgs bundles $(V,\, \theta)$ on $M$
such that $c_1(V)\,=\, 0\, =\, c_2(V)$, and $(\phi^*V,\, \phi^*\theta)$
is polystable; the morphisms are homomorphisms of Higgs bundles.
\end{enumerate}
\end{theorem}

\begin{proof}
The homomorphism of fundamental groups induced by $\phi$
$$
\phi_*\, :\, \pi_1(X,\, x_0)\, \longrightarrow\, \pi_1(M,\, \phi(x_0))
$$
is an isomorphism, because $\phi$ is bimeromorphic. Consequently, the operation of
pullback, to $X$, of flat vector bundles on $M$ identifies the flat bundles on $M$ with those
on $X$. Also, connection preserving homomorphisms between two flat bundles on $M$ coincide
with connection preserving homomorphisms between their pullback to $X$.

Let $(V_1,\, \theta_1)$ be a polystable Higgs bundle on $X$ with
$c_1(V_1)\,=\,0\,=\, c_2(V_1)$. Then, as shown in the proof
of Theorem \ref{thmfm}, using Proposition \ref{propositionpullback} the vector bundle $V_1$ descends to
a holomorphic vector bundle on $M$, meaning there exists a bundle $W_1$ on $M$ such that $V_1\,=\, \phi^*W_1$.
Since $c_1(V_1)\,=\,0\,=\, c_2(V_1)$, this implies that $c_1(W_1)\,=\,0\,=\, c_2(W_1)$.

Let $U\, \subset\, X$ be the open subset over which $\phi$ is an isomorphism.

The Higgs field $\theta_1$ defines a Higgs field on $W_1\vert_{\phi(U)}$.
Again using Hartogs' extension theorem this Higgs field on $W_1\vert_{\phi(U)}$ extends to a Higgs
field on $W_1$; this extended Higgs field on $W_1$ will be denoted by $\theta'$. It is evident that
$(\phi^*W_1,\, \phi^*\theta')\,=\, (V_1,\, \theta_1)$.

If $(W_2,\, \theta'')$ is a polystable Higgs bundle on $M$ with
$c_1(W_2)\,=\,0\,=\, c_2(W_2)$, then it can be shown that
\begin{equation}\label{st}
H^0(M,\, \text{Hom}((W_1,\, \theta'),\, (W_2,\, \theta'')))\,=\, H^0(X,\,
\text{Hom}((V_1,\, \theta_1),\, (\phi^*W_2,\, \phi^*\theta'')))\, .
\end{equation}
To prove this, let $(E_1,\, \nabla_1)$ (respectively, $(E_2,\, \nabla_2)$) be the completely reducible
flat complex connection on $M$ corresponding to $(W_1,\, \theta')$ (respectively, $(W_2,\, \theta'')$). Then
$$
H^0(X,\, \text{Hom}((V_1,\, \theta_1),\, (\phi^*W_2,\, \phi^*\theta'')))\,=\,
H^0(X,\, \text{Hom}((\phi^*E_1,\, \phi^*\nabla_1),\, (\phi^*E_2,\, \phi^*\nabla_2)))\, .
$$
But
$$
H^0(X,\, \text{Hom}((\phi^*E_1,\, \phi^*\nabla_1),\, (\phi^*E_2,\, \phi^*\nabla_2)))\,=\,
H^0(M,\, \text{Hom}((E_1,\, \nabla_1),\, (E_2,\, \nabla_2)))\, .
$$
Hence from the isomorphism
$$
H^0(M,\, \text{Hom}((E_1,\, \nabla_1),\, (E_2,\, \nabla_2)))\,=\,
H^0(M,\, \text{Hom}((W_1,\, \theta'),\, (W_2,\, \theta'')))
$$
we conclude that \eqref{st} holds. This completes the proof.
\end{proof}

It is rather routine to check that the results on \cite{BG} extend to the context
of Fujiki class $\mathcal C$ manifolds.

\section*{Acknowledgements}

We are grateful to Andreas H\"oring who helped us providing the proof of Proposition \ref{propositionpullback}.
We also thank Yohan Brunebarbe and Carlos Simpson for useful conversations on the subject.

This work has been supported by the French government through the UCAJEDI Investments in the 
Future project managed by the National Research Agency (ANR) with the reference number 
ANR2152IDEX201. The first-named author is partially supported by a J. C. Bose Fellowship, and 
school of mathematics, TIFR, is supported by 12-R$\&$D-TFR-5.01-0500. 


\end{document}